\documentclass[11pt]{article}
\usepackage{enumerate}
\usepackage{lineno,hyperref}
\usepackage{amsfonts}
\usepackage{amssymb}
\usepackage{mathrsfs}
\usepackage{srcltx}
\usepackage{amsmath}
\usepackage{amsthm}
\usepackage{amstext}
\usepackage{amsopn}
\usepackage{graphicx}
\modulolinenumbers[4]

\newtheorem{theorem}{Theorem}[section]
\newtheorem{lemma}[theorem]{Lemma}
\newtheorem{proposition}[theorem]{Proposition}
\newtheorem{corollary}[theorem]{Corollary}

\theoremstyle{definition}

\theoremstyle{remark}
\newtheorem{remark}[theorem]{Remark}

\numberwithin{equation}{section}

\newcommand{\ba}{\begin{array}}
\newcommand{\ea}{\end{array}}

\newcommand{\R}{\mathbb{R}}

\newcommand{\N}{\mathbb{N}}

\allowdisplaybreaks[4]

\begin{document}
 \date{}
 \title{ Berestycki-Lions conditions on ground state solutions for Kirchhoff-type problems with variable potentials
\footnote{This paper was submitted to Journal on April 18, 2018.}}
\author{ Sitong Chen \ and Xianhua Tang\\
        {\small School of Mathematics and Statistics, Central South University,}\\
        {\small Changsha 410083, Hunan, P.R.China }\\
        {\small E-mail: mathsitongchen@163.com (S.T. Chen)}\\
        {\small E-mail: tangxh@mail.csu.edu.cn(X.H. Tang)}}
\maketitle
\begin{abstract}
 By introducing some new tricks, we prove that the nonlinear problem of Kirchhoff-type
 \begin{equation*}
 \left\{
   \begin{array}{ll}
     -\left(a+b\int_{\R^3}|\nabla u|^2\mathrm{d}x\right)\triangle u+V(x)u=f(u), & x\in \R^3; \\
     u\in H^1(\R^3),
   \end{array}
 \right.
 \end{equation*}
 admits two class of ground state solutions under the general ``Berestycki-Lions assumptions" on the nonlinearity $f$
 which are almost necessary conditions, as well as  some weak assumptions on the potential $V$. Moreover, we also give a
 simple minimax characterization of the ground state energy. Our results improve and complement previous ones in the literature.


\end{abstract}

{\section{Introduction}}
 \setcounter{equation}{0}

   In this paper, we consider the following nonlinear problem of Kirchhoff-type:
 \begin{equation}\label{KE}
 \left\{
   \begin{array}{ll}
     -\left(a+b\int_{\R^3}|\nabla u|^2\mathrm{d}x\right)\triangle u+V(x)u=f(u), & x\in \R^3; \\
     u\in H^1(\R^3),
   \end{array}
 \right.
 \end{equation}
 where $a, b>0$ are two constants, $V: \R^3\rightarrow \R$ and $f: \R\rightarrow \R$ satisfy

 \vskip2mm
 \begin{itemize}
 \item[(V1)] $V\in \mathcal{C}(\R^3, [0, \infty))$;
 \item[(V2)]  $V_{\infty}:=\liminf_{|y|\to \infty}V(y)\ge  V(x)$ for all $x\in \R^3$;
 \item[(F1)]  $f\in \mathcal{C}(\R, \R)$ and there exists a constant $\mathcal{C}_0>0$ such that
 $$
   |f(t)|\le \mathcal{C}_0\left(1+|t|^{5}\right), \ \ \ \ \forall \ t\in \R;
 $$

 \item[(F2)]  $f(t)=o(t)$ as $t\to 0$ and $|f(t)|=o\left(|t|^{5}\right)$ as $|t|\to +\infty$.
 \end{itemize}

   Clearly, under assumptions (V1), (V2), (F1) and (F2), weak solutions to \eqref{KE} correspond to critical points of
 the energy functional defined in $H^1(\R^3)$ by
 \begin{equation}\label{IU}
   \mathcal{I}(u) = \frac{1}{2}\int_{\R^3}\left[a|\nabla u|^2+V(x)u^2\right]\mathrm{d}x+\frac{b}{4}\left(\int_{\R^3}|\nabla u|^2\mathrm{d}x\right)^2
              -\int_{\R^3}F(u)\mathrm{d}x,
 \end{equation}
 where and in the sequel, $F(t):=\int_{0}^{t}f(s)\mathrm{d}s$. We say a nontrivial weak solution $\bar{u}$ to \eqref{KE} is a
 ground state solution if $\mathcal{I}(\bar{u})\le \mathcal{I}(v)$ for any nontrivial solution $v$ to \eqref{KE}.

 \par
   There have been many works about the existence of nontrivial solutions to \eqref{KE} by using variational methods, see for example,
 \cite{AC,CC,CJ, CBT,DY,FG,HZ1,HZ2, HY,LCY,LYH, LP,ND1,ND2,PK,SW,SJJ, TC2,WJ,WX,ZTC} and the references therein. A typical way
 to deal with \eqref{KE} is to use the mountain-pass theorem. For this purpose, one usually assumes that
 $f(t)$ is subcritical and superlinear at $t=0$ and 4-superlinear at $t=\infty$ in the sense that

 \begin{itemize}
 \item[(SF)] $\lim_{|t|\to\infty}\frac{F(t)}{t^4}=\infty$,
 \end{itemize}
 and satisfies the Ambrosetti-Rabinowitz type condition
 \begin{itemize}
 \item[(AR)] $f(t)t\ge 4F(t)\ge 0, \ \forall \ t\in \R$;
 \end{itemize}
 or the following variant convex condition
 \begin{itemize}
  \item[(S1)]$f(t)/|t|^3$ is strictly increasing for $t\in \R\setminus \{0\}$.
 \end{itemize}
 In fact, under (SF) and (AR) (or (S1)), it is easy to verify the Mountain Pass geometry and the boundedness of (PS)
 sequences for $\mathcal{I}$.

 \par
   When $f(t)$ is not 4-superlinear at $t=\infty$, following the procedure of Ruiz \cite{RD} in which the nonlinear
 Schr\"odinger-Poisson system was dealt with, Li and Ye \cite{LG} first proved that the following special form of \eqref{KE}
 \begin{equation}\label{KE3}
 \left\{
   \begin{array}{ll}
     -\left(a+b\int_{\R^3}|\nabla u|^2\mathrm{d}x\right)\triangle u+u=|u|^{p-2}u, & x\in \R^3; \\
     u\in H^1(\R^3),
   \end{array}
 \right.
 \end{equation}
 has a ground state positive solution if $3 < p <6$, by using a minimizing argument on a Nehari- Poho\u zaev manifold obtained
 by combining the Nehari manifold and the corresponding Poho\u zaev identity. Subsequently,
 by introducing a new Nehari-Poho\u zaev manifold differing from \cite{LG} and using Jeanjean's monotonicity trick \cite{Je}
 and a suitable approximating method, Guo \cite{Gu} generalized Li and Ye's result to \eqref{KE}, where $V$ and $f$ satisfy
 (V1), (V2), (F1), (F2) and
 \begin{itemize}
 \item[(V3$'$)] $V\in \mathcal{C}^1(\R^3, \R)$ and there exists $\theta\in (0, 1)$ such that
 $$
   |\nabla V(x)\cdot x|\le  \frac{\theta a}{2|x|^2}, \ \ \ \ \forall \ x\in \R^3\setminus \{0\};
 $$
 \item[(S2)]  $f\in \mathcal{C}^{1}(\R^{+}, \R)$ and $\left(\frac{f(t)}{t}\right)'>0$.
 \end{itemize}
 Applying Guo's result to \eqref{KE3}, the condition $3 < p <6$ in \cite{LG} can be relaxed to $2 < p <6$.
 More recently, Tang and Chen \cite{TC} introduced some new skills to weaken (V3$'$) and (S2) to the following conditions
 \begin{itemize}
 \item[(V3)] $V\in \mathcal{C}^1(\R^3, \R)$, and
 $$
   \nabla V(x)\cdot x\le  \frac{a}{2|x|^2}, \ \ \ \ \forall \ x\in \R^3\setminus \{0\};
 $$
 \item[(S3)] $f\in \mathcal{C}(\R, \R)$ and $\frac{f(t)t+6F(t)}{|t|t}$ is nondecreasing on $(-\infty, 0)\cup(0, \infty)$.
 \end{itemize}

 \par
   We remark that (SF), (AR), (S1)-(S3) are all global growth conditions. Inspired by the fundamental paper \cite{BL}, Azzollini
 \cite{Az} proved that the ``limit problem" associated with \eqref{KE}
 \begin{equation}\label{KE4}
 \left\{
   \begin{array}{ll}
     -\left(a+b\int_{\R^3}|\nabla u|^2\mathrm{d}x\right)\triangle u+V_{\infty}u=f(u), & x\in \R^3; \\
     u\in H^1(\R^3),
   \end{array}
 \right.
 \end{equation}
 has a ground state positive solution if $f$ satisfies the Berestycki-Lions type assumptions: (F1), (F2) and the following local assumption

 \begin{itemize}
 \item[(F3)] there exists $s_0>0$ such that $F(s_0)>\frac{1}{2}V_{\infty}s_0^2$.
 \end{itemize}

 \par
    Obviously, (F1)-(F3) are satisfied by a very wide class of nonlinearities. In particular, only local conditions on $f(t)$
 are required. Moreover, in view of \cite{Az}, (F1)-(F3) are ``almost" necessary for the existence
 of a nontrivial solution of problem \eqref{KE4}. This kind of conditions were first introduced by Berestycki and Lions \cite{BL}
 for the study of the nonlinear scalar field equation
 \begin{equation}\label{SE}
    -\triangle u+V_{\infty}u=f(u), \ \ \ \  u\in H^1(\R^N).
 \end{equation}

 \par
   To prove the above result, Azzollini considered the following constrained minimization problem
 $$
 m^{\infty}:=\inf_{u\in \mathcal{M}^{\infty}}\mathcal{I}^{\infty}(u),
 $$
 where
 \begin{equation}\label{Ii}
   \mathcal{I}^{\infty}(u):=\frac{1}{2}\int_{{\R}^3}\left(a|\nabla u|^2+V_{\infty}u^2\right)\mathrm{d}x
      +\frac{b}{4}\left(\int_{\R^3}|\nabla u|^2\mathrm{d}x\right)^2
            -\int_{{\R}^3}F(u)\mathrm{d}x
 \end{equation}
 is the energy functional associated with \eqref{KE4}, and
 $$
   \mathcal{M}^{\infty}:=\{u\in H^1(\R^3)\setminus \{0\}: \mathcal{P}^{\infty}(u)=0\}
 $$
 is the Poho\u zaev manifold, and $\mathcal{P}^{\infty}$ is the Poho\u zaev functional defined by
 \begin{eqnarray}\label{Pi}
  \mathcal{P}^{\infty}(u)
    &  =  & \frac{a}{2}\|\nabla u\|_2^2+\frac{3}{2}V_{\infty}\|u\|_2^2+\frac{b}{2}\|\nabla u\|_2^4-3\int_{{\R}^3}F(u)\mathrm{d}x.
 \end{eqnarray}

 \par
    Azzollini first proved that $\mathcal{I}^{\infty}$ possesses a minimizer $u_{\infty}$ on $\mathcal{M}^{\infty}\cap H^1_r(\R^3)$,
 it is also a minimizer on $\mathcal{M}^{\infty}$ by Schwarz symmetrization, then verified that $u_{\infty}$ is a critical point of $\mathcal{I}^{\infty}$ by means of the Lagrange multipliers Theorem.

 \par
   In another paper \cite{Az1}, Azzollini, by means of a rescaling argument, established a general relationship between solutions
 of \eqref{KE4} and \eqref{SE}. That is $u\in \mathcal{C}^2(\R^3)\cap \mathcal{D}^{1,2}(\R^3)$ is a solution to \eqref{KE4} if
 and only if there exist $v\in \mathcal{C}^2(\R^3)\cap \mathcal{D}^{1,2}(\R^3)$ satisfying \eqref{SE} and $t>0$ such that
 $t^2a+tb\|\nabla v\|_2^2=1$ and $u(x)=v(tx)$. With this relationship and the results obtained in \cite{BL,JT1} in hand, Azzollini
 \cite{Az1} also concluded the same results as \cite{Az}. Following \cite{Az1}, Lu \cite{Lu} proved that \eqref{KE4} has infinitely
 many distinct radial solutions if $f$ is odd and satisfies (F1)-(F3).

 \par
   The approach used in \cite{Az, Az1} is valid only for autonomous equations, it does not work any more for nonautonomous equation
 \eqref{KE} with $V\ne$ constant. In the present paper, based on \cite{Az,BL,JT2,TC3}, we shall develop a new approach
 to look for a ground state solution for \eqref{KE} by using (F3) instead of (S3). Our results improve and generalize
 the Azzollini's results in \cite{Az,Az1} on autonomous equation \eqref{KE4}. More precisely, we have the following theorem.

 \begin{theorem}\label{thm1.1} Assume that $V$ and $f$ satisfy {\rm(V1)-(V3)} and {\rm(F1)-(F3)}. Then
 problem \eqref{KE} has a ground state solution.
 \end{theorem}

 \par
    To prove Theorem \ref{thm1.1}, we will use an idea from Jeanjean and Tanaka \cite{JT2}, that is an approximation procedure to obtain
 a bounded (PS)-sequence for $\mathcal{I}$, instead of starting directly from an arbitrary (PS)-sequence.
 More precisely, firstly for $\lambda\in [1/2, 1]$ we consider a family of functionals
  $\mathcal{I}_{\lambda} : H^1(\R^3) \rightarrow \R$ defined by
  \begin{equation}\label{Ilu}
   \mathcal{I}_{\lambda}(u)=\frac{1}{2}\int_{{\R}^3}\left[a|\nabla u|^2+V(x)u^2\right]\mathrm{d}x
            +\frac{b}{4}\left(\int_{\R^3}|\nabla u|^2\mathrm{d}x\right)^2
            -\lambda\int_{{\R}^3}F(u)\mathrm{d}x.
 \end{equation}
 These functionals have a Mountain Pass geometry, and denoting the corresponding Mountain Pass levels by $c_{\lambda}$.
 Let
 $$
   A(u)=\frac{1}{2}\int_{\R^3}\left[a|\nabla u|^2+V(x)u^2\right]\mathrm{d}x+\frac{b}{4}\left(\int_{\R^3}|\nabla u|^2\mathrm{d}x\right)^2, \ \ \ \
   B(u)=\int_{\R^3}F(u)\mathrm{d}x.
 $$
 Then $I_{\lambda}(u)=A(u)-\lambda B(u)$. Unfortunately, $B(u)$ is not sign definite under (F1)-(F3), it prevents us
 from employing Jeanjean's monotonicity trick \cite{Je} used in \cite{JT2}. Thanks to the work of Jeanjean and Toland \cite{JTo},
 $\mathcal{I}_{\lambda}$ still has a bounded (PS)-sequence $\{u_n(\lambda)\} \subset H^1(\R^3)$ at level $c_{\lambda}$ for almost
 every $\lambda\in [1/2,1]$. However, there is no more a monotone dependence of $c_{\lambda}$ upon $\lambda\in [1/2,1]$ in this case,
 while it plays a crucial role in Jeanjean's monotonicity trick. To show that the bounded sequence $\{u_n(\lambda)\}$ converges
 weakly to a nontrivial critical point of $\mathcal{I}_{\lambda}$, one usually has to establish the following strict inequality
 \begin{eqnarray}\label{ck0}
   c_{\lambda}< \inf_{\mathcal{K}_\lambda^{\infty}}\mathcal{I}_{\lambda}^{\infty},
 \end{eqnarray}
 where
 \begin{equation}\label{Il}
   \mathcal{I}_{\lambda}^{\infty}(u)=\frac{1}{2}\int_{{\R}^3}\left(a|\nabla u|^2+V_{\infty}u^2\right)\mathrm{d}x
       +\frac{b}{4}\left(\int_{\R^3}|\nabla u|^2\mathrm{d}x\right)^2-\lambda\int_{{\R}^3}F(u)\mathrm{d}x
 \end{equation}
 and
 \begin{eqnarray}\label{KI}
   \mathcal{K}_\lambda^{\infty}:=\left\{u\in H^1(\R^3)\setminus \{0\} : \ (\mathcal{I}_{\lambda}^{\infty})'(u)=0\right\}.
 \end{eqnarray}
 In view of the results in \cite{Az, Az1}, for every $\lambda\in [1/2, 1]$, there exists $w_{\lambda}^{\infty}\in \mathcal{K}_\lambda^{\infty}$
 such that $\mathcal{I}_{\lambda}^{\infty}(w_{\lambda}^{\infty})=\inf_{\mathcal{K}_\lambda^{\infty}}\mathcal{I}_{\lambda}^{\infty}$.
 Since $V(x)\le V_{\infty}$ but $V(x)\not\equiv V_{\infty}$, it is standard to show \eqref{ck0} if $w_{\lambda}^{\infty}>0$.
 However, there is no more information on the sign of $w_{\lambda}^{\infty}$ from the results in \cite{Az, Az1}. Therefore,
 it becomes nontrivial to show \eqref{ck0}. To overcome this difficulty we use a strategy introduced in \cite{TC3}. Let
 \begin{eqnarray}\label{JlL}
  \mathcal{P}_{\lambda}^{\infty}(u)
    =\frac{a}{2}\|\nabla u\|_2^2+\frac{3V_\infty}{2}\int_{{\R}^3}u^2\mathrm{d}x
    +\frac{b}{2}\left(\int_{\R^3}|\nabla u|^2\mathrm{d}x\right)^2-3\lambda\int_{{\R}^3}F(u)\mathrm{d}x
 \end{eqnarray}
 and
 \begin{eqnarray}\label{Ml}
   \mathcal{M}_{\lambda}^{\infty}:=\left\{u\in H^1(\R^3)\setminus \{0\}: \mathcal{P}_{\lambda}^{\infty}(u)=0\right\}.
 \end{eqnarray}
 We first prove that problem \eqref{KE4} has a solution $\bar{u}^{\infty}\in H^1(\R^3)$ such that $\mathcal{I}^{\infty}(\bar{u}^{\infty})=\inf_{\mathcal{M}^{\infty}}\mathcal{I}^{\infty}$.
 By means of the translation invariance for $\bar{u}^{\infty}$ and a crucial inequality related to $\mathcal{I}(u)$, $\mathcal{I}(u_t)$ and
 $\mathcal{P}(u)$ (the IIP inequality in short, see Lemma \ref{lem 2.2}, where $u_t(x)=u(x/t)$, it plays an important
 role in many places of this paper), we can find $\bar{\lambda}\in [1/2, 1)$ and prove directly the following crucial inequality
 \begin{eqnarray}\label{cm1}
   c_{\lambda}<m_{\lambda}^{\infty}:=\inf_{\mathcal{M}_{\lambda}^{\infty}}\mathcal{I}_{\lambda}^{\infty},
   \ \ \ \ \lambda\in (\bar{\lambda}, 1].
 \end{eqnarray}
 In particular, it is not required any information on sign of $\bar{u}^{\infty}$ in our arguments.
 Then applying \eqref{cm1} and a precise decomposition of bounded (PS)-sequences, we can get a nontrivial
 critical point $u_{\lambda}$ of $\mathcal{I}_{\lambda}$ which possesses energy $c_{\lambda}$  for almost every
 $\lambda\in (\bar{\lambda}, 1]$.

 \par
   In the proof of Theorem \ref{thm1.1}, a crucial step is to show that problem \eqref{KE4} has a solution $\bar{u}^{\infty}
 \in H^1(\R^3)$ such that $\mathcal{I}^{\infty}(\bar{u}^{\infty})=\inf_{\mathcal{M}^{\infty}}\mathcal{I}^{\infty}$. With the
 help of the Lions' concentration compactness, the IIP inequality established in Lemma \ref{lem 2.2}, the ``least energy squeeze
 approach" and some subtle analysis, we can prove a more general conclusion. In fact, we shall conclude that \eqref{KE} has a
 solution $\bar{u}\in \mathcal{M}$ such that $\mathcal{I}(\bar{u})=\inf_{\mathcal{M}}\mathcal{I}$ if $f$ satisfies (F1)-(F3)
 and $V$ satisfies (V1), (V2) and the following decay assumption on $V$:
 \begin{itemize}
 \item[(V4)] $V\in \mathcal{C}^1(\R^3, \R)$ and $t\mapsto 3V(tx)+\nabla V(tx)\cdot (tx)+\frac{a}{4t^{2}|x|^2}$ is nonincreasing on
 $(0,\infty)$ for every $x\in \R^3\setminus \{0\}$;
 \end{itemize}
 where
 \begin{equation}\label{Ne}
   \mathcal{M}:= \{u\in H^1(\R^3)\setminus \{0\} : \mathcal{P}(u)=0\}
 \end{equation}
 and
 \begin{eqnarray}\label{Jv}
   \mathcal{P}(u) & :=  & \frac{a}{2}\|\nabla u\|_2^2+\frac{1}{2}\int_{{\R}^3}[3V(x)+\nabla V(x)\cdot x]u^2\mathrm{d}x\nonumber\\
        &      & \ \ +\frac{b}{2}\|\nabla u\|_2^4-3\int_{{\R}^3}F(u)\mathrm{d}x.
 \end{eqnarray}
 Actually the equality $\mathcal{P}(u) = 0$ is nothing but the Poho\u zaev identity related with equation \eqref{KE}.
 More precisely, we have the following theorem.

 \begin{theorem}\label{thm1.2} Assume that $V$ and $f$ satisfy {\rm(V1), (V2), (V4)} and {\rm(F1)-(F3)}. Then
 problem \eqref{KE} has a solution $\bar{u}\in H^1(\R^3)$ such that $\mathcal{I}(\bar{u})=\inf_{\mathcal{M}}\mathcal{I}
 =\inf_{u\in \Lambda}\max_{t > 0}\mathcal{I}(u_t)>0$, where
 $$
   u_t(x):=u(x/t) \ \ \mbox{and} \ \
  \Lambda ：=\left\{u\in H^1(\R^3) : \int_{\R^3}\left[\frac{1}{2}V_{\infty}u^2-F(u)\right]\mathrm{d}x<0\right\}.
 $$
 \end{theorem}

 \par
   As a consequence of Theorem \ref{thm1.2}, we have the following corollary.

 \begin{corollary}\label{cor1.3} Assume that $f$ satisfies {\rm(F1)-(F3)}. Then problem \eqref{KE4} has a solution
 $\bar{u}\in H^1(\R^3)$ such that $\mathcal{I}^{\infty}(\bar{u})=\inf_{\mathcal{M}^{\infty}}\mathcal{I}^{\infty}
 =\inf_{u\in \Lambda}\max_{t > 0}\mathcal{I}^{\infty}(u_t)>0$.
 \end{corollary}

 \begin{remark}\label{rem1.5} As a consequence of Theorem \ref{thm1.2}, the ground state value $m:=\inf_{\mathcal{M}}\mathcal{I}$
 has a minimax  characterization $ m=\inf_{u\in \Lambda}\max_{t > 0}\mathcal{I}(u_t)$ which is much simpler than the usual
 characterizations related to the Mountain Pass level.
 \end{remark}

 \par
   Our approach to show Theorem \ref{thm1.2} is different from the ones used in \cite{Az, Az1}. Moreover, Theorem \ref{thm1.2}
 generalizes the Azzollini's results in \cite{Az,Az1} on autonomous equation \eqref{KE4} to \eqref{KE} with $V\ne$ constant.
 In particular, such an approach could be useful for the study of other problems where radial symmetry of bounded sequence
 either fails or is not readily available.

 \begin{remark}\label{rem1.1}
 There are indeed many functions which satisfy {\rm (V1)-(V3)}. For example
 \par
 i). $V(x) =\alpha-\frac{\beta}{|x|^{\sigma}+1}$ with $\alpha>\beta>0$, $\sigma\ge 2$ and $a\ge 2\sigma\beta$;
 \par
 ii). $V(x) =\alpha-\frac{\beta\sin^2 |x|}{|x|^{3}+1}$ with $\alpha>\beta>0$ and $a\ge 4\beta$;
 \par
 iii). $V(x) =\alpha-\beta e^{-|x|^{\sigma}}$ with $\alpha>\beta>0$, $\sigma>0$  and $ae^{(\sigma+2)/\sigma}\ge 2\beta(\sigma+2)^{(\sigma+2)/\sigma}\sigma^{-2/\sigma}$.

 \par\noindent
 In particular, if $\alpha>\beta>0$, $\sigma\ge 2$ and $a\ge 2\sigma\beta(3+\sigma)$, then $V(x) =\alpha-\frac{\beta}{|x|^{\sigma}+1}$
 also satisfies (V4).

 \end{remark}

   Applying Theorem \ref{thm1.1} to the following perturbed problem:
 \begin{equation}\label{KE8}
 \left\{
   \begin{array}{ll}
     -\left(a+b\int_{\R^3}|\nabla u|^2\mathrm{d}x\right)\triangle u+[V_{\infty}-\varepsilon h(x)]u=f(u), & x\in \R^3; \\
     u\in H^1(\R^3),
   \end{array}
 \right.
 \end{equation}
 where $V_{\infty}$ is a positive constant and the function $h \in \mathcal{C}^1(\R^3, \R)$ verifies:

 \vskip2mm
 \noindent
 (H1)\ \ $h(x)\ge 0$ for all $x\in \R^3$ and $\lim_{|x|\to\infty}h(x)=0$;

 \vskip2mm
 \noindent
 (H2)\ \ $\sup_{x\in \R^3}\left[-|x|^2\nabla h(x)\cdot x\right]<\infty$.

 \vskip2mm
 \noindent
 Then we have the following corollary.

 \begin{corollary}\label{cor1.4} Assume that $h$ and $f$ satisfy {\rm(H1), (H2)} and {\rm(F1)-(F3)}. Then there
 exists a constant $\varepsilon_0>0$ such that problem \eqref{KE8} has a ground state solution $\bar{u}_{\varepsilon}\in
 H^1(\R^3)\setminus \{0\}$ for all $0<\varepsilon\le \varepsilon_0$.
 \end{corollary}

   Throughout the paper we make use of the following notations:

 \vskip4mm
 \par
     $\spadesuit$ \ $H^1(\R^3)$ denotes the usual Sobolev space equipped with the inner product and norm
 $$
   (u, v)=\int_{\R^3}(\nabla u\cdot \nabla v+uv)\mathrm{d}x, \ \ \|u\|=(u, u)^{1/2},
     \ \ \forall \ u,v\in H^1(\R^3);
 $$

 \par
     $\spadesuit$ \ $H_r^1(\R^3)=\{u\in H^1(\R^3): |x|=|y|\Rightarrow u(x)=u(y)\}$;

 \par
     $\spadesuit$ \ $L^s(\R^3) (1\le s< \infty)$  denotes the Lebesgue space with the norm $\|u\|_s
 =\left(\int_{\R^3}|u|^s\mathrm{d}x\right)^{1/s}$;

 \par
     $\spadesuit$ \ For any $u\in H^1(\R^3)\setminus \{0\}$, $u_t(x):=u(t^{-1}x)$ for $t>0$;

 \par
     $\spadesuit$ \ For any $x\in \R^3$ and $r>0$, $B_r(x):=\{y\in \R^3: |y-x|<r \}$;

 \par
     $\spadesuit$ \ $C_1, C_2,\cdots$ denote positive constants possibly different in different places.

 \vskip4mm
   The rest of the paper is organized as follows. In Section 2, we give some preliminaries, and give the proof
 of Theorem \ref{thm1.2}. In Section 3, we complete the proof of Theorem \ref{thm1.1}.

 \vskip6mm

 {\section{Proof of Theorem \ref{thm1.2}}}
 \setcounter{equation}{0}

 \setcounter{equation}{0}
 \vskip2mm
 \par
   In this section, we give the proof of Theorem \ref{thm1.2}. To this end, we give some useful lemmas. Since $V(x)\equiv V_{\infty}$
 satisfies (V1), (V2) and (V4), thus all conclusions on $\mathcal{I}$ are also true for $\mathcal{I}^{\infty}$.  For \eqref{KE4},
 we always assume that $V_{\infty}>0$. First, by a simple calculation, we can verify Lemma \ref{lem 2.1}.

  \begin{lemma}\label{lem 2.1}
 Assume that {\rm (V4)} holds. Then one has
 \begin{equation}\label{B20}
   3t^3[V(x)-V(tx)]-(1-t^3)\nabla V(x)\cdot x \ge  -\frac{a(1-t)^2(2+t)}{4|x|^2}, \ \ \ \ \forall \ t\ge 0,
      \ \ x\in \R^3\setminus \{0\}.
 \end{equation}
\end{lemma}

 \begin{lemma}\label{lem 2.2}
 Assume that {\rm (V1), (V2), (V4), (F1)} and {\rm (F2)} hold. Then
 \begin{eqnarray}\label{B21}
   \mathcal{I}(u)
    & \ge & \mathcal{I}\left(u_t\right)+\frac{1-t^3}{3}\mathcal{P}(u)+\frac{b(1-t)^2(1+2t)}{12}\|\nabla u\|_2^4,
             \ \ \ \ \forall \ u\in H^1(\R^3), \ \ t > 0. \ \ \ \
 \end{eqnarray}
\end{lemma}

\begin{proof} According to Hardy inequality, we have
 \begin{equation}\label{B22}
  \|\nabla u\|_2^2 \ge \frac{1}{4}\int_{\R^3}\frac{u^2}{|x|^2}\mathrm{d}x,  \ \ \forall \ u\in H^1(\R^3).
 \end{equation}
 Note that
 \begin{equation}\label{B23}
   \mathcal{I}\left(u_t\right) = \frac{at}{2}\|\nabla u\|_2^2+\frac{t^3}{2}\int_{{\R}^3}V(tx)u^2\mathrm{d}x
                   +\frac{bt^2}{4}\|\nabla u\|_2^4 -t^3\int_{{\R}^3}F(u)\mathrm{d}x.
 \end{equation}
 Thus, by \eqref{IU}, \eqref{Jv}, \eqref{B20}, \eqref{B22} and \eqref{B23}, one has
 \begin{eqnarray*}
   &     & \mathcal{I}(u)-\mathcal{I}\left(u_t\right)\\
   &  =  & \frac{a(1-t)}{2}\|\nabla u\|_2^2+\frac{1}{2}\int_{{\R}^3}\left[V(x)-t^3V(tx)\right]u^2\mathrm{d}x
             +\frac{b(1-t^2)}{4}\|\nabla u\|_2^4\\
   &     & +(t^3-1)\int_{\R^3}F(u)\mathrm{d}x\\
   &  =  & \frac{1-t^3}{3}\left\{\frac{a}{2}\|\nabla u\|_2^2+\frac{1}{2}\int_{{\R}^3}[3V(x)+\nabla V(x)\cdot x]u^2\mathrm{d}x
            +\frac{b}{2}\|\nabla u\|_2^4-3\int_{{\R}^3}F(u)\mathrm{d}x\right\}\\
   &     & +\frac{a(1-t)^2(2+t)}{6}\|\nabla u\|_2^2+\frac{1}{6}\int_{{\R}^3}\left\{3t^3[V(x)-V(tx)]
             -(1-t^3)\nabla V(x)\cdot x\right\}u^2\mathrm{d}x\\
   &     &   +\frac{b(1-t)^2(1+2t)}{12}\|\nabla u\|_2^4\\
   & \ge & \frac{1-t^3}{3}\mathcal{P}(u)+\frac{b(1-t)^2(1+2t)}{12}\|\nabla u\|_2^4.
 \end{eqnarray*}
 This shows that \eqref{B21} holds.
 \end{proof}

 \vskip4mm
 \par
   From Lemma \ref{lem 2.2}, we have the following two corollaries.

\begin{corollary}\label{cor 2.3}
 Assume that {\rm (F1)} and {\rm (F2)} hold. Then
 \begin{eqnarray}\label{B25}
   \mathcal{I}^{\infty}(u)
     & \ge & \mathcal{I}^{\infty}\left(u_t\right)+\frac{1-t^3}{3}\mathcal{P}^{\infty}(u)+\frac{b(1-t)^2(1+2t)}{12}\|\nabla u\|_2^4,\nonumber\\
     &     &  \ \ \ \ \forall \ u\in H^1(\R^3), \ \ t > 0.
 \end{eqnarray}

\end{corollary}

 \begin{corollary}\label{cor 2.4}
 Assume that {\rm (V1), (V2), (V4), (F1)} and {\rm (F2)} hold. Then for $u\in \mathcal{M}$
 \begin{equation}\label{Imax}
   \mathcal{I}(u) = \max_{t> 0}\mathcal{I}\left(u_t\right).
 \end{equation}
 \end{corollary}

 \begin{lemma}\label{lem 2.5}
 Assume that {\rm (V1), (V2)} and {\rm (V4)} hold. Then there exist two constants $\gamma_1, \gamma_2>0$
 such that
 \begin{equation}\label{B26}
   \gamma_1\|u\|_2^2\le a\|\nabla u\|_2^2+\int_{{\R}^3}\left[3V(x)+\nabla V(x)\cdot x\right]u^2\mathrm{d}x\le \gamma_2\|u\|^2,
     \ \ \forall \ u\in H^1(\R^3).
 \end{equation}
\end{lemma}

\begin{proof} Let $t \to \infty$ in \eqref{B20}, and using (V2), one has
 \begin{equation}\label{DV}
   3V(x)+\nabla V(x)\cdot x\le 3V_{\infty}+\frac{a}{2|x|^2},
       \ \ \ \ \forall \ x\in \R^3\setminus \{0\}.
 \end{equation}
 Then it follows from \eqref{B22} and \eqref{DV} that there exists $\gamma_2>0$ such that the second inequality in \eqref{B26} holds.

 Next, we prove that the first inequality holds.
 By \eqref{B20}, one has
 \begin{equation}\label{D1}
  3V(x)+\nabla V(x)\cdot x\ge 3V(tx)-\frac{a}{4|x|^2}\left(1-\frac{1}{t}\right)^2\left(1+\frac{2}{t}\right),
   \ \ \ \ \forall \ t> 0, \ x\in \R^3\setminus \{0\}.
 \end{equation}
 It is easy to see that there exist $\varepsilon_0>0$ and $t_0>0$ such that
 \begin{equation*}
  \left(1-\frac{1}{t_0}\right)^2\left(1+\frac{2}{t_0}\right)\le 1-\varepsilon_0,
 \end{equation*}
 which, together with \eqref{D1}, implies
 \begin{equation}\label{D2}
  3V(x)+\nabla V(x)\cdot x\ge 3V(t_0x)-\frac{a}{4|x|^2}(1-\varepsilon_0),
   \ \ \ \ \forall \ x\in \R^3\setminus \{0\}.
 \end{equation}
 By (V2), there exists $R_0>0$ such that $V(x)\ge V_{\infty}/2$ for all $x\ge |t_0R_0|$.
 Choose $\alpha_0\in (0, V_{\infty}/2)$ such that
 \begin{equation}\label{D3}
   \varepsilon_0a-3\alpha_0\left(\frac{4\pi R_0}{3}\right)^{2/3}S^{-1}\ge \frac{\varepsilon_0a}{2}.
 \end{equation}
 Then it follows from (V1), \eqref{B22}, \eqref{D2}, \eqref{D3} and Sobolev inequality that
 \begin{eqnarray}\label{D5}
   &     & a\|\nabla u\|_2^2+\int_{{\R}^3}\left[3V(x)+\nabla V(x)\cdot x\right]u^2\mathrm{d}x\nonumber\\
   & \ge & a\|\nabla u\|_2^2+\int_{{\R}^3}\left[3V(t_0x)-\frac{a}{4|x|^2}(1-\varepsilon_0)\right]u^2\mathrm{d}x\nonumber\\
   & \ge & a\|\nabla u\|_2^2+3\alpha_0\int_{|x|\ge R_0}u^2\mathrm{d}x-(1-\varepsilon_0)a\|\nabla u\|_2^2\nonumber\\
   &  =  & 3\alpha_0\|u\|_2^2+\varepsilon_0a\|\nabla u\|_2^2-3\alpha_0\int_{|x|\le R_0}u^2\mathrm{d}x\nonumber\\
   & \ge & 3\alpha_0\|u\|_2^2+\varepsilon_0a\|\nabla u\|_2^2-3\alpha_0\left(\frac{4\pi R_0}{3}\right)^{2/3}
            \left(\int_{|x|\le R_0}u^6\mathrm{d}x\right)^{1/3}\nonumber\\
   & \ge & 3\alpha_0\|u\|_2^2+\varepsilon_0a\|\nabla u\|_2^2-3\alpha_0\left(\frac{4\pi R_0}{3}\right)^{2/3}
            S^{-1}\|\nabla u\|_2^2\nonumber\\
   & \ge & \min\left\{3\alpha_0, \frac{\varepsilon_0a}{2}\right\}\|u\|^2:=\gamma_1\|u\|^2, \ \ \forall \ u\in H^1(\R^3).
 \end{eqnarray}
 \end{proof}

 \vskip4mm
 \par
   To show $\mathcal{M}\ne \emptyset$, we define a set $\Lambda$ as follows:
 \begin{equation}\label{La}
   \Lambda：=\left\{u\in H^1(\R^3) : \int_{\R^3}\left[\frac{1}{2}V_{\infty}u^2-F(u)\right]\mathrm{d}x<0\right\}.
 \end{equation}

 \begin{lemma}\label{lem 2.6}
 Assume that {\rm (V1), (V2), (V4)} and {\rm (F1)-(F3)} hold. Then $\Lambda\ne\emptyset$ and
 \begin{equation}\label{La1}
   \left\{u\in H^1(\R^3)\setminus \{0\} : \mathcal{P}^{\infty}(u)\le 0 \ \mbox{or} \ \mathcal{P}(u)\le 0\right\}\subset \Lambda.
 \end{equation}
 \end{lemma}

 \begin{proof} In view of the proof of \cite[Theorem 2]{BL}, (F3) implies $\Lambda\ne\emptyset$. Next, we have two cases
 to distinguish:

 \vskip2mm
 \par
  1). $u\in H^1(\R^3)\setminus \{0\}$ and $\mathcal{P}^{\infty}(u)\le 0$, then \eqref{Pi} implies $u\in \Lambda$.

 \par
  2). Let $t=0$ and $t \to \infty$ in \eqref{B20}, respectively, and using (V2), one has
 \begin{equation}\label{B27}
   -\frac{a}{4|x|^2}+3V_{\infty} \le 3V(x)+\nabla V(x)\cdot x\le 3V_{\infty}+\frac{a}{2|x|^2},
       \ \ \ \ \forall \ x\in \R^N\setminus \{0\}.
 \end{equation}
 For $u\in H^1(\R^3)\setminus \{0\}$ and $\mathcal{P}(u)\le 0$, then it follows from \eqref{Jv}, \eqref{B22} and \eqref{B27} that
 \begin{eqnarray*}
   &     &  3\int_{\R^3}\left[\frac{1}{2}V_{\infty}u^2-F(u)\right]\mathrm{d}x\\
   &  =  & \mathcal{P}(u)-\frac{a}{2}\|\nabla u\|_2^2-\frac{1}{2}\int_{{\R}^N}\left[3(V(x)-V_{\infty})
             +\nabla V(x)\cdot x\right]u^2\mathrm{d}x-\frac{b}{2}\|\nabla u\|_2^4\\
   & \le & -\frac{a}{2}\|\nabla u\|_2^2+\frac{a}{8}\int_{{\R}^3}\frac{u^2}{|x|^2}\mathrm{d}x-\frac{b}{2}\|\nabla u\|_2^4\\
   & \le & -\frac{b}{2}\|\nabla u\|_2^4 < 0,
 \end{eqnarray*}
 which implies $u\in \Lambda$.
 \end{proof}

 \begin{lemma}\label{lem 2.7}
 Assume that {\rm (V1), (V2), (V4)} and {\rm (F1)-(F3)} hold. Then for any
 $u\in \Lambda$, there exists a unique $t_u>0$ such that $u_{t_u}\in \mathcal{M}$.
 \end{lemma}

 \begin{proof}  Let $u\in \Lambda$ be fixed and define a function $\zeta(t):=\mathcal{I}\left(u_t\right)$
 on $(0, \infty)$. Clearly, by \eqref{Jv} and \eqref{B23}, we have
 \begin{eqnarray*}
   \zeta'(t)=0
    &     &  \Leftrightarrow
            \ \ \frac{a}{2}\|\nabla u\|_2^2+\frac{t^2}{2}\int_{{\R}^3}[3V(tx)+\nabla V(tx)\cdot tx]u^2\mathrm{d}x\nonumber\\
    &     & \ \ \ \ \ \ +\frac{bt}{2}\|\nabla u\|_2^4-3t^2\int_{{\R}^3}F\left(u\right)\mathrm{d}x=0\\
    &     & \ \Leftrightarrow  \ \ \mathcal{P}(u_t)=0 \ \ \Leftrightarrow  \ \ u_t\in \mathcal{M}.
 \end{eqnarray*}
 It is easy to verify, using (V1), (V2) and the definition of $\Lambda$, that $\lim_{t\to 0}\zeta(t)=0$, $\zeta(t)>0$
 for $t>0$ small and $\zeta(t)<0$ for $t$ large. Therefore $\max_{t\in (0, \infty)}\zeta(t)$ is achieved at $t_u>0$ so that
 $\zeta'(t_u)=0$ and $u_{t_u}\in \mathcal{M}$.

 \par
    Now we pass to prove that $t_u$ is unique for any $u\in \Lambda$. In fact, for any given $u\in \Lambda$, let $t_1, t_2>0$
 such that $u_{t_1}, u_{t_2} \in \mathcal{M}$. Then $\mathcal{P}\left(u_{t_1}\right)=\mathcal{P}\left(u_{t_2}\right)=0$.
 Jointly with \eqref{B21}, we have
 \begin{eqnarray}\label{B41}
   \mathcal{I}\left(u_{t_1}\right)
   & \ge & \mathcal{I}\left(u_{t_2}\right)+\frac{t_1^3-t_2^3}{3t_1^3}\mathcal{P}\left(u_{t_1}\right)
             +\frac{b(t_1^2-t_2^2)^2(t_1+2t_2)}{12t_1}\|\nabla u\|_2^4\nonumber\\
   &  =  & \mathcal{I}\left(u_{t_2}\right)+\frac{b(t_1^2-t_2^2)^2(t_1+2t_2)}{12t_1}\|\nabla u\|_2^4
 \end{eqnarray}
 and
 \begin{eqnarray}\label{B42}
   \mathcal{I}\left(u_{t_2}\right)
   & \ge & \mathcal{I}\left(u_{t_1}\right)+\frac{t_2^3-t_1^3}{3t_2^3}\mathcal{P}\left(u_{t_2}\right)
             +\frac{b(t_2^2-t_1^2)^2(t_2+2t_1)}{12t_2}\|\nabla u\|_2^4\nonumber\\
   &  =  & \mathcal{I}\left(u_{t_1}\right)+\frac{b(t_2^2-t_1^2)^2(t_2+2t_1)}{12t_2}\|\nabla u\|_2^4.
 \end{eqnarray}
 \eqref{B41} and \eqref{B42} imply $t_1=t_2$. Therefore, $t_u> 0$ is unique for any $u\in \Lambda$.
 \end{proof}

 \begin{corollary}\label{cor 2.8}
 Assume that {\rm (F1)-(F3)} hold. Then for any $u\in \Lambda$, there exists a unique $t_u>0$ such that $u_{t_u}\in \mathcal{M}^{\infty}$.
 \end{corollary}

    Combining Corollary \ref{cor 2.4} with Lemma \ref{lem 2.7}, we have the following lemma.

 \begin{lemma}\label{lem 2.9}
 Assume that {\rm (V1), (V2), (V4)} and {\rm (F1)-(F3)} hold. Then
 $$
   \inf_{u\in \mathcal{M}}\mathcal{I}(u) :=m=\inf_{u\in \Lambda}\max_{t > 0}\mathcal{I}\left(u_t\right).
 $$
 \end{lemma}

 \par
   Similar to \cite[Lemma 2.10]{TC2}, we have the following lemma.

 \begin{lemma}\label{lem 2.10}
 Assume that {\rm (V1), (V2), (F1)} and {\rm (F2)} hold. If $u_n\rightharpoonup \bar{u}$ in $H^1(\R^3)$, then
 \begin{equation}\label{F60}
  \mathcal{I}(u_n)=\mathcal{I}(\bar{u})+\mathcal{I}(u_n-\bar{u})+\frac{b}{2}\|\nabla\bar{u}\|_2^2\|\nabla (u_n-\bar{u})\|_2^2+o(1)
 \end{equation}
 and
 \begin{equation}\label{F63}
   \mathcal{P}(u_n)=\mathcal{P}(\bar{u})+\mathcal{P}(u_n-\bar{u})+b\|\nabla\bar{u}\|_2^2\|\nabla (u_n-\bar{u})\|_2^2+o(1).
 \end{equation}
\end{lemma}

\begin{lemma}\label{lem 2.11}
 Assume that {\rm (V1), (V2), (V4)} and {\rm (F1)-(F3)} hold. Then
 \begin{enumerate}[{\rm (i)}]
  \item there exists $\rho_0>0$ such that $\|\nabla u\|_2\ge \rho_0, \ \forall \ u\in \mathcal{M}$;
  \item $m=\inf_{u\in \mathcal{M}}\mathcal{I}(u)>0$.
 \end{enumerate}
\end{lemma}

\begin{proof} (i). Since $\mathcal{P}(u)=0, \ \forall u\in \mathcal{M}$, by (F1), (F2), \eqref{Jv}, \eqref{B26} and Sobolev
 embedding inequality $S\|u\|_6^2\le \|\nabla u\|_2^2$, one has
 \begin{eqnarray}\label{G62}
   \gamma_1\|u\|^2+b\|\nabla u\|_2^4
   & \le & a\|\nabla u\|_2^2+\int_{{\R}^3}\left[3V(x)+\nabla V(x)\cdot x\right]u^2\mathrm{d}x
             +b\|\nabla u\|_2^4\nonumber\\
   &  =  & 6\int_{{\R}^3}F(u)\mathrm{d}x \nonumber\\
   & \le & \frac{\gamma_1}{2}\|u\|_2^2+C_1\|u\|_6^6\le \frac{\gamma_1}{2}\|u\|_2^2+C_1S^{-3}\|\nabla u\|_2^6,
 \end{eqnarray}
 which implies
 \begin{equation}\label{G63}
   \|\nabla u\|_2\ge \rho_0:=\left(\frac{bS^3}{C_1}\right)^{1/2}, \ \ \ \ \forall \ u\in \mathcal{M}.
 \end{equation}

 \par
   (ii). By \eqref{B21} with $t\to 0$, we have
 \begin{equation}\label{G64}
  \mathcal{I}(u)=\mathcal{I}(u)-\frac{1}{3}\mathcal{P}(u)\ge \frac{b}{12}\|\nabla u\|_2^4,
             \ \ \ \ \forall \ u\in \mathcal{M}.
 \end{equation}
 This, together with \eqref{G63} shows that $m=\inf_{u\in \mathcal{M}}\mathcal{I}(u)>0$.
 \end{proof}

 \begin{lemma}\label{lem 2.12}
 Assume that {\rm (V1), (V2), (V4)} and {\rm (F1)-(F3)} hold. Then $m\le m^{\infty}$.
 \end{lemma}

 \begin{proof} In view of Lemma \ref{lem 2.6} and Corollary \ref{cor 2.8}, we have $\mathcal{M}^{\infty}\ne \emptyset$.
 Arguing indirectly, we assume that $m> m^{\infty}$. Let $\varepsilon:=m-m^{\infty}$. Then there exists $u_{\varepsilon}^{\infty}$
 such that
 \begin{equation}\label{P11}
   u_{\varepsilon}^{\infty}\in \mathcal{M}^{\infty} \ \ \ \  \mbox{and}
     \ \ \ \ m^{\infty}+\frac{\varepsilon}{2}>\mathcal{I}^{\infty}(u_{\varepsilon}^{\infty}).
 \end{equation}
 In view of Lemmas \ref{lem 2.6} and \ref{lem 2.7}, there exists $t_{\varepsilon}>0$ such that $(u_{\varepsilon}^{\infty})_{t_{\varepsilon}}\in \mathcal{M}$. Thus,
 it follows from (V2), \eqref{IU}, \eqref{Ii}, \eqref{B25} and \eqref{P11} that
 $$
   m^{\infty}+\frac{\varepsilon}{2}>\mathcal{I}^{\infty}(u_{\varepsilon}^{\infty})
     \ge \mathcal{I}^{\infty}\left((u_{\varepsilon}^{\infty})_{t_{\varepsilon}}\right)
     \ge \mathcal{I}\left((u_{\varepsilon}^{\infty})_{t_{\varepsilon}}\right)\ge m.
 $$
 This contradiction shows the conclusion of Lemma \ref{lem 2.12} is true.
 \end{proof}

\begin{lemma}\label{lem 2.13}
 Assume that {\rm (V1), (V2), (V4)} and {\rm (F1)-(F3)} hold. Then $m$ is achieved.
\end{lemma}

 \begin{proof} In view of Lemmas \ref{lem 2.6}, \ref{lem 2.7} and \ref{lem 2.11}, we have $\mathcal{M}\ne \emptyset$
 and $m>0$. Let $\{u_n\}\subset \mathcal{M}$ be such that $\mathcal{I}(u_n)\rightarrow m$.
 Since $\mathcal{P}(u_n)=0$, then it follows from \eqref{B21} with $t \rightarrow 0$, we have
 \begin{equation}\label{P24}
   m+o(1)= \mathcal{I}(u_n)\ge \frac{b}{12}\|\nabla u_n\|_2^4.
 \end{equation}
 This shows that $\{\|\nabla u_n\|_2\}$ is bounded. Next, we prove that $\{\|u_n\|\}$ is also bounded.
 From (F1), (F2), \eqref{Jv}, \eqref{B26} and Sobolev  embedding inequality, one has
 \begin{eqnarray}\label{P25}
   \gamma_1\|u_n\|^2
   & \le & a\|\nabla u_n\|_2^2+\int_{{\R}^3}\left[3V(x)+\nabla V(x)\cdot x\right]u_n^2\mathrm{d}x
             +b\|\nabla u_n\|_2^4\nonumber\\
   &  =  & 6\int_{{\R}^3}F(u)\mathrm{d}x \nonumber\\
   & \le & \frac{\gamma_1}{2}\|u\|_2^2+C_3\|u\|_6^6\le \frac{\gamma_1}{2}\|u\|_2^2+C_3S^{-3}\|\nabla u\|_2^6.
 \end{eqnarray}
 Hence, $\{u_n\}$ is bounded in $H^1(\R^3)$. Passing to a subsequence, we have
 $u_n\rightharpoonup \bar{u}$ in $H^1(\R^3)$. Then $u_n\rightarrow \bar{u}$ in $L_{\mathrm{loc}}^s(\R^3)$
 for $2\le s<2^*$ and $u_n\rightarrow \bar{u}$ a.e. in $\R^3$. There are two possible cases: i). $\bar{u}=0$
 and ii). $\bar{u}\ne 0$.

 \vskip2mm
 \par
   Case i). $\bar{u}=0$, i.e. $u_n\rightharpoonup 0$ in $H^1(\R^3)$.
 Then $u_n\rightarrow 0$ in $L_{\mathrm{loc}}^s(\R^3)$ for $2\le s<2^*$ and $u_n\rightarrow 0$ a.e. in $\R^3$.
 By (V2) and \eqref{B27}, it is easy to show that
 \begin{equation}\label{P32}
   \lim_{n\to\infty}\int_{\R^3}[V_{\infty}-V(x)]u_n^2\mathrm{d}x=
   \lim_{n\to\infty}\int_{\R^3}\nabla V(x)\cdot xu_n^2\mathrm{d}x=0.
 \end{equation}
 From \eqref{IU}, \eqref{Ii}, \eqref{Pi}, \eqref{Jv} and \eqref{P32}, one can get
 \begin{equation}\label{P37}
   \mathcal{I}^{\infty}(u_n)\rightarrow m, \ \ \ \ \mathcal{P}^{\infty}(u_n)\rightarrow 0.
 \end{equation}
 From Lemma \ref{lem 2.11} (i), \eqref{Pi} and \eqref{P37}, one has
 \begin{eqnarray}\label{P64}
  a\rho_0^2\le a\|\nabla u_n\|_2^2+3V_{\infty}\|u_n\|_2^2+\frac{b}{2}\|\nabla u_n\|_2^4=6\int_{\R^3}F(u_n)\mathrm{d}x+o(1).
 \end{eqnarray}
 Using (F1), (F2), \eqref{P64} and Lions' concentration compactness principle \cite[Lemma 1.21]{WM}, we can prove that
 there exist $\delta>0$ and a sequence $\{y_n\}\subset \R^3$ such that $\int_{B_1(y_n)}|u_n|^2\mathrm{d}x> \delta$. Let
 $\hat{u}_n(x)=u_n(x+y_n)$. Then we have $\|\hat{u}_n\|=\|u_n\|$ and
 \begin{equation}\label{P65}
   \mathcal{P}^{\infty}(\hat{u}_n)= o(1), \ \ \ \ \mathcal{I}^{\infty}(\hat{u}_n)\rightarrow m,
     \ \ \ \ \int_{B_1(0)}|\hat{u}_n|^2\mathrm{d}x> \delta.
 \end{equation}
 Therefore, there exists $\hat{u}\in H^1(\R^3)\setminus \{0\}$ such that, passing to a subsequence,
 \begin{equation}\label{P72}
 \left\{
   \begin{array}{ll}
     \hat{u}_n\rightharpoonup \hat{u}, & \mbox{in} \ H^1(\R^3); \\
     \hat{u}_n\rightarrow \hat{u}, & \mbox{in} \ L_{\mathrm{loc}}^s(\R^3), \ \forall \ s\in [1, 6);\\
     \hat{u}_n\rightarrow \hat{u}, & \mbox{a.e. on} \ \R^3.
   \end{array}
 \right.
 \end{equation}
 Let $w_n=\hat{u}_n-\hat{u}$. Then \eqref{P72} and Lemma \ref{lem 2.10} yield
 \begin{equation}\label{D73}
    \mathcal{I}^{\infty}(\hat{u}_n) = \mathcal{I}^{\infty}(\hat{u})+\mathcal{I}^{\infty}(w_n)
        +\frac{b}{2}\|\nabla\hat{u}\|_2^2\|\nabla w_n\|_2^2+o(1)
 \end{equation}
 and
 \begin{equation}\label{D74}
    \mathcal{P}^{\infty}(\hat{u}_n) = \mathcal{P}^{\infty}(\hat{u})+\mathcal{P}^{\infty}(w_n)+b\|\nabla\hat{u}\|_2^2\|\nabla w_n\|_2^2+o(1).
 \end{equation}
 Set
 \begin{equation}\label{Ps0}
   \Psi_0(u)=\frac{4a\|\nabla u\|_2^2+b\|\nabla u\|_2^4}{12}.
 \end{equation}
 Then one has,
 \begin{equation}\label{P75}
    \Psi_0(w_n)\le m-\Psi_0(\hat{u})+o(1),
      \ \ \ \ \mathcal{P}^{\infty}(w_n) \le -\mathcal{P}^{\infty}(\hat{u})+o(1).
 \end{equation}
 If there exists a subsequence $\{w_{n_i}\}$ of $\{w_n\}$ such that $w_{n_i}=0$, then going to this subsequence, we have
 \begin{equation}\label{P76}
    \mathcal{I}^{\infty}(\hat{u})=m, \ \ \ \ \mathcal{P}^{\infty}(\hat{u})=0.
 \end{equation}
 Next, we assume that $w_n\ne 0$. We claim that $\mathcal{P}^{\infty}(\hat{u})\le 0$. Otherwise, if $\mathcal{P}^{\infty}(\hat{u})>0$,
 then \eqref{P75} implies $\mathcal{P}^{\infty}(w_n) < 0$ for large $n$. In view of Lemma \ref{lem 2.6} and Corollary \ref{cor 2.8}, there exists $t_n>0$ such that $(w_n)_{t_n}\in \mathcal{M}^{\infty}$. From \eqref{Ii}, \eqref{Pi}, \eqref{B25}, \eqref{Ps0} and \eqref{P75}, we obtain
 \begin{eqnarray*}
   m-\Psi_0(\hat{u})+o(1)
    & \ge & \Psi_0(w_n) = \mathcal{I}^{\infty}(w_n)-\frac{1}{3}\mathcal{P}^{\infty}(w_n)\nonumber\\
    & \ge & \mathcal{I}^{\infty}\left({(w_n)}_{t_n}\right)-\frac{t_n^3}{3}\mathcal{P}^{\infty}(w_n)\nonumber\\
    & \ge & m^{\infty}-\frac{t_n^3}{3}\mathcal{P}^{\infty}(w_n)\ge m^{\infty},
 \end{eqnarray*}
 which implies $\mathcal{P}^{\infty}(\hat{u})\le 0$ due to $m\le m^{\infty}$ and $\Psi_0(\hat{u})>0$. Since $\hat{u}\ne 0$ and
 $\mathcal{P}^{\infty}(\hat{u})\le 0$, in view of Lemma \ref{lem 2.6} and Corollary \ref{cor 2.8}, there exists
 $t_{\infty}>0$ such that $\hat{u}_{t_{\infty}}\in \mathcal{M}^{\infty}$.  From \eqref{Ii}, \eqref{Pi}, \eqref{B25}, \eqref{P65},  \eqref{Ps0}
 and the weak semicontinuity of norm, one has
 \begin{eqnarray*}
   m &  =  & \lim_{n\to\infty} \left[\mathcal{I}^{\infty}(\hat{u}_n)-\frac{1}{3}\mathcal{P}^{\infty}(\hat{u}_n)\right]\nonumber\\
     &  =  & \lim_{n\to\infty}\Psi_0(\hat{u}_n)\ge \Psi_0(\hat{u})\nonumber\\
     &  =  & \mathcal{I}^{\infty}(\hat{u})-\frac{1}{3}\mathcal{P}^{\infty}(\hat{u})\ge \mathcal{I}^{\infty}\left({\hat{u}}_{t_{\infty}}\right)
               -\frac{t_{\infty}^3}{3}\mathcal{P}^{\infty}(\hat{u})\nonumber\\
     & \ge & m^{\infty}-\frac{t_{\infty}^3}{3}\mathcal{P}^{\infty}(\hat{u})\nonumber\\
     & \ge & m-\frac{t_{\infty}^3}{3}\mathcal{P}^{\infty}(\hat{u})\ge m,
 \end{eqnarray*}
 which implies \eqref{P76} holds also. In view of Lemmas \ref{lem 2.6} and \ref{lem 2.7}, there exists $\hat{t}>0$ such that
 $\hat{u}_{\hat{t}}\in \mathcal{M}$, moreover, it follows from (V2), \eqref{IU}, \eqref{Ii}, \eqref{P76} and Corollary \ref{cor 2.3} that
 $$
   m\le \mathcal{I}(\hat{u}_{\hat{t}})\le \mathcal{I}^{\infty}(\hat{u}_{\hat{t}})\le \mathcal{I}^{\infty}(\hat{u})=m.
 $$
 This shows that $m$ is achieved at $\hat{u}_{\hat{t}}\in \mathcal{M}$.

 \vskip2mm
 \par
   Case ii). $\bar{u}\ne 0$. Let $v_n=u_n-\bar{u}$. Then Lemma \ref{lem 2.10} yields
 \begin{equation}\label{K71}
    \mathcal{I}(u_n)=\mathcal{I}(\bar{u})+\mathcal{I}(v_n)+\frac{b}{2}\|\nabla\bar{u}\|_2^2\|\nabla v_n\|_2^2+o(1)
 \end{equation}
 and
 \begin{equation}\label{K72}
   \mathcal{P}(u_n)=\mathcal{P}(\bar{u})+\mathcal{P}(v_n)+b\|\nabla\bar{u}\|_2^2\|\nabla v_n\|_2^2+o(1).
 \end{equation}
 Set
 \begin{equation}\label{K73}
   \Psi(u)=\frac{4a\|\nabla u\|_2^2+b\|\nabla u\|_2^4}{12}-\frac{1}{6}\int_{{\R}^3}\nabla V(x)\cdot xu^2\mathrm{d}x.
 \end{equation}
 Then it follows from \eqref{B22} and \eqref{B20} with $t=0$ that
 \begin{eqnarray}\label{K74}
  \Psi(u)
   &  =  & \frac{4a\|\nabla u\|_2^2+b\|\nabla u\|_2^4}{12}-\frac{1}{6}\int_{{\R}^3}\nabla V(x)\cdot xu^2\mathrm{d}x\nonumber\\
   & \ge & \frac{4a\|\nabla u\|_2^2+b\|\nabla u\|_2^4}{12}-\frac{a}{12}\int_{{\R}^3}\frac{u^2}{|x|^2}\mathrm{d}x\nonumber\\
   & \ge & \frac{b}{12}\|\nabla u\|_2^4, \ \ \ \ \forall \ u\in H^1(\R^3).
 \end{eqnarray}
 Since $\mathcal{I}(u_n)\rightarrow m$ and $\mathcal{P}(u_n)=0$, then it follows from \eqref{IU}, \eqref{Jv}, \eqref{K71}, \eqref{K72} and \eqref{K73} that
 \begin{equation}\label{K75}
    \Psi(v_n)\le m-\Psi(\bar{u})+o(1), \ \ \ \ \mathcal{P}(v_n) \le -\mathcal{P}(\bar{u})+o(1).
 \end{equation}
 If there exists a subsequence $\{v_{n_i}\}$ of $\{v_n\}$ such that $v_{n_i}=0$, then going to this subsequence, we have
 \begin{equation}\label{K76}
    \mathcal{I}(\bar{u})=m, \ \ \ \ \mathcal{P}(\bar{u})=0,
 \end{equation}
 which implies the conclusion of Lemma \ref{lem 2.13} holds. Next, we assume that $v_n\ne 0$. We claim that $\mathcal{P}(\bar{u})\le 0$.
 Otherwise $\mathcal{P}(\bar{u})>0$, then \eqref{K75} implies $\mathcal{P}(v_n) < 0$ for large $n$. In view of Lemmas \ref{lem 2.6} and \ref{lem 2.7},
 there exists $t_n>0$ such that $(v_n)_{t_n}\in \mathcal{M}$.  From \eqref{IU}, \eqref{Jv}, \eqref{B21}, \eqref{K73} and \eqref{K75}, we obtain
 \begin{eqnarray*}
   m-\Psi(\bar{u})+o(1)
    & \ge & \Psi(v_n) = \mathcal{I}(v_n)-\frac{1}{3}\mathcal{P}(v_n)\nonumber\\
    & \ge & \mathcal{I}\left({(v_n)}_{t_n}\right)-\frac{t_n^3}{3}\mathcal{P}(v_n)\nonumber\\
    & \ge & m-\frac{t_n^3}{3}\mathcal{P}(v_n)\ge m,
 \end{eqnarray*}
 which implies $\mathcal{P}(\bar{u})\le 0$ due to $\Psi(\bar{u})>0$. Since $\bar{u}\ne 0$ and $\mathcal{P}(\bar{u})\le 0$,
 in view of Lemmas \ref{lem 2.6} and \ref{lem 2.7}, there exists $\bar{t}>0$ such that $\bar{u}_{\bar{t}}\in \mathcal{M}$.
 From \eqref{IU}, \eqref{Jv}, \eqref{B21}, \eqref{K73}, \eqref{K74} and the weak semicontinuity of norm, one has
 \begin{eqnarray*}
   m
   &  =  & \lim_{n\to\infty} \left[\mathcal{I}(u_n)-\frac{1}{3}\mathcal{P}(u_n)\right]\nonumber\\
   &  =  & \lim_{n\to\infty} \Psi(u_n)\ge \Psi(\bar{u})\nonumber\\
   &  =  & \mathcal{I}(\bar{u})-\frac{1}{3}\mathcal{P}(\bar{u})\ge \mathcal{I}\left({\bar{u}}_{\bar{t}}\right)
             -\frac{\bar{t}^3}{3}\mathcal{P}(\bar{u})\nonumber\\
   & \ge & m-\frac{\bar{t}^3}{3}\mathcal{P}(\bar{u})\ge m,
 \end{eqnarray*}
 which implies \eqref{K76} also holds.
 \end{proof}

 \begin{lemma}\label{lem 2.14}
 Assume that {\rm (V1), (V2), (V4)} and {\rm (F1)-(F3)} hold. If $\bar{u}\in \mathcal{M}$
 and $\mathcal{I}(\bar{u})=m$, then $\bar{u}$ is a critical point of $\mathcal{I}$.
 \end{lemma}

 \begin{proof}  Similar to the proof of \cite[Lemma 2.13]{TC3}, we can prove this lemma only by using
 \begin{eqnarray}\label{B90}
   \mathcal{I}\left(\bar{u}_t\right)
   & \le & \mathcal{I}(\bar{u})-\frac{b(1-t)^2(1+2t)}{12}\|\nabla \bar{u}\|_2^4\nonumber\\
   &  =  & m-\frac{b(1-t)^2(1+2t)}{12}\|\nabla \bar{u}\|_2^4, \ \ \ \ \forall \ t> 0.
 \end{eqnarray}
 and
 $$
   \varepsilon:=\min\left\{\frac{b(1-T_1)^2(1+2T_1)\|\nabla\bar{u}\|_2^4}{36}, \frac{b(1-T_2)^2(1+2T_2)\|\nabla\bar{u}\|_2^4}{36}, 1, \frac{\varrho\delta}{8}\right\}
 $$
 instead of \cite[(2.40) and $\varepsilon$] {TC3}, respectively.
 \end{proof}

 \begin{proof}[Proof of Theorem \ref{thm1.2}] In view of Lemmas \ref{lem 2.9}, \ref{lem 2.13} and \ref{lem 2.14},
 there exists $\bar{u}\in \mathcal{M}$ such that
 $$
   \mathcal{I}(\bar{u})=m=\inf_{u\in \Lambda}\max_{t > 0}\mathcal{I}(u_t), \ \ \ \ \mathcal{I}'(\bar{u})=0.
 $$
 This shows that $\bar{u}$ is a ground state solution of \eqref{KE}.
 \end{proof}

 \vskip6mm
 {\section{Proof of Theorem \ref{thm1.1}}}
 \setcounter{equation}{0}

 \vskip2mm
 \par
   In this section, we assume that $V(x)\not\equiv V_{\infty}$ and give the proof of Theorem \ref{thm1.1}.
 \begin{proposition}\label{pro 3.1}{\rm \cite{JTo}}
 Let $X$ be a Banach space and let $J\subset \R^+$ be an interval, and
 $$
   \Phi_{\lambda}(u)=A(u)-\lambda B(u), \ \ \ \ \forall \ \lambda\in J,
 $$
 be a family of $\mathcal{C}^1$-functional on $X$ such that
 \begin{enumerate}[{\rm(i)}]
  \item  either $A(u)\to +\infty$ or $B(u)\to +\infty$, as $\|u\| \to \infty$;
  \item  $B$ maps every bounded set of $X$ into a set of $\R$ bounded below;
  \item  there are two points $v_1, v_2$ in $X$ such that
 \begin{equation}\label{cm}
   \tilde{c}_{\lambda}:=\inf_{\gamma\in \Gamma}\max_{t\in [0, 1]}\Phi_{\lambda}(\gamma(t))>\max\{\Phi_{\lambda}(v_1), \Phi_{\lambda}(v_2)\},
 \end{equation}
 \end{enumerate}
 where
 $$
   \Gamma=\left\{\gamma\in \mathcal{C}([0, 1], X): \gamma(0)=v_1, \gamma(1)=v_2\right\}.
 $$
 Then, for almost every $\lambda\in J$, there exists a sequence  such that
 \begin{enumerate}[{\rm (i)}]
  \item $\{u_n(\lambda)\}$ is bounded in $X$;
  \item $\Phi_{\lambda}(u_n(\lambda))\rightarrow c_{\lambda}$;
  \item $\Phi_{\lambda}'(u_n(\lambda))\rightarrow 0$ in $X^*$, where $X^*$ is the dual of $X$.
 \end{enumerate}
 \end{proposition}

 \begin{lemma}\label{lem 3.2}{\rm \cite{Gu}}
 Assume that {\rm (V1)-(V3), (F1)} and {\rm (F2)} hold. Let $u$ be a critical point of
 $\mathcal{I}_{\lambda}$ in $H^1(\R^3)$, then we have the following Poho\u zaev type identity
 \begin{eqnarray}\label{Pl}
  \mathcal{P}_{\lambda}(u)
    & :=  & \frac{a}{2}\|\nabla u\|_2^2+\frac{1}{2}\int_{{\R}^3}\left[3V(x)+\nabla V(x)\cdot x\right]u^2\mathrm{d}x\nonumber\\
    &     & \ \  +\frac{b}{2}\left(\int_{\R^3}|\nabla u|^2\mathrm{d}x\right)^2-3\lambda\int_{{\R}^3}F(u)\mathrm{d}x=0.
 \end{eqnarray}
 \end{lemma}

 Correspondingly, we also let
 \begin{eqnarray}\label{PlL}
  \mathcal{P}_{\lambda}^{\infty}(u)
    =\frac{a}{2}\|\nabla u\|_2^2+\frac{3V_\infty}{2}\int_{{\R}^3}u^2\mathrm{d}x
    +\frac{b}{2}\left(\int_{\R^3}|\nabla u|^2\mathrm{d}x\right)^2-3\lambda\int_{{\R}^3}F(u)\mathrm{d}x,
 \end{eqnarray}
 for $\lambda\in [1/2, 1]$. Set
 \begin{equation}\label{cm0}
  \mathcal{M}_{\lambda}^{\infty}:=\{u\in H^1(\R^3)\setminus \{0\}: \mathcal{P}_{\lambda}^{\infty}(u)=0\}, \ \ \ \
   m_{\lambda}^{\infty}:=\inf_{u\in \mathcal{M}_{\lambda}^{\infty}}\mathcal{I}_{\lambda}^{\infty}(u).
 \end{equation}

 \vskip4mm
 \par
    By Corollary \ref{cor 2.3}, we have the following lemma.

 \begin{lemma}\label{lem 3.3}
 Assume that {\rm (F1)} and {\rm (F2)}  hold. Then
 \begin{eqnarray}\label{N45}
   \mathcal{I}_{\lambda}^{\infty}(u)
    & \ge & \mathcal{I}_{\lambda}^{\infty}\left(u_t\right)+\frac{1-t^3}{3}\mathcal{P}_{\lambda}^{\infty}(u)
              +\frac{b(1-t)^2(1+2t)}{12}\|\nabla u\|_2^4,\nonumber\\
    &     &  \ \ \ \ \ \ \ \ \forall \ u\in H^1(\R^3), \ \ t > 0.
 \end{eqnarray}
 \end{lemma}

 \begin{lemma}\label{lem 3.4}
 Assume that {\rm (V1)-(V3)} and {\rm (F1)-(F3)} hold. Then
 \begin{enumerate}[{\rm (i)}]
  \item there exists $T>0$ independent of $\lambda$ such that $\mathcal{I}_{\lambda}\left((u_1^{\infty})_{T}\right)<0$
 for all $\lambda\in [0.5, 1]$;
  \item there exists a positive constant $\kappa_0 $ independent of $\lambda$ such that for all $\lambda\in [0.5, 1]$,
 \begin{equation*}
   c_{\lambda}:=\inf_{\gamma\in \Gamma}\max_{t\in [0, 1]}\mathcal{I}_{\lambda}(\gamma(t))\ge \kappa_0
     >\max\left\{\mathcal{I}_{\lambda}(0), \mathcal{I}_{\lambda}\left((u_1^{\infty})_{T}\right)\right\},
 \end{equation*}
 where
 $$
   \Gamma=\left\{\gamma\in \mathcal{C}([0, 1], H^1(\R^3)): \gamma(0)=0, \gamma(1)=(u_1^{\infty})_{T}\right\};
 $$
 \item $c_{\lambda}$ is bounded for $\lambda\in [0.5, 1]$;
 \item $m_{\lambda}^{\infty}$ is non-increasing on $\lambda\in [0.5, 1]$;
 \item $\limsup_{\lambda\to \lambda_0}c_{\lambda}\le c_{\lambda_0}$ for $\lambda_0\in(0.5, 1]$.
 \end{enumerate}
\end{lemma}

   Since $m_{\lambda}^{\infty}=\mathcal{I}_{\lambda}^{\infty}(u_{\lambda}^{\infty})$ and
 $\int_{\R^3}F(u_{\lambda}^{\infty})\mathrm{d}x>0$, then the proof of (i)-(iv) in Lemma \ref{lem 3.4} is
 standard, (v) can be proved similar to \cite[Lemma 2.3]{Je}, so we omit it.

    In view of Corollary \ref{cor1.3}, $\mathcal{I}_1^{\infty}=\mathcal{I}^{\infty}$ has a minimizer $u_1^{\infty}\ne 0$
 on $\mathcal{M}_1^{\infty} =\mathcal{M}^{\infty}$,  i.e.
 \begin{equation}\label{N47}
   u_1^{\infty}\in \mathcal{M}_1^{\infty}, \ \ \ \ (\mathcal{I}_1^{\infty})'(u_1^{\infty})=0
     \ \ \ \  \mbox{and} \ \ \ \ m_1^{\infty}=\mathcal{I}_1^{\infty}(u_1^{\infty}),
 \end{equation}
 where $m_{\lambda}^{\infty}$ is defined by \eqref{cm0}. Since \eqref{KE4} is autonomous, $V\in \mathcal{C}(\R^3, \R)$ and
 $V(x)\le V_{\infty}$ but $V(x)\not\equiv V_{\infty}$, then there exist $\bar{x}\in \R^3$ and $\bar{r}>0$ such that
 \begin{equation}\label{N48}
    V_{\infty}-V(x)>0, \ \ |u_1^{\infty}(x)|>0\ \ \ \ a.e. \ |x-\bar{x}|\le \bar{r}.
 \end{equation}

 \begin{lemma}\label{lem 3.5}
 Assume that {\rm (V1)-(V3)} and {\rm (F1)-(F3)} hold. Then
 there exists $\bar{\lambda}\in [1/2, 1)$ such that $c_{\lambda}<m_{\lambda}^{\infty}$ for $\lambda\in (\bar{\lambda}, 1]$.
\end{lemma}

 \begin{proof} It is easy to see that $\mathcal{I}_{\lambda}\left((u_1^{\infty})_t\right)$ is continuous on $t\in (0, \infty)$.
 Hence for any $\lambda\in [1/2, 1]$,  we can choose $t_{\lambda}\in (0, T)$ such that $\mathcal{I}_{\lambda}
 \left((u_1^{\infty})_{t_{\lambda}}\right) =\max_{t\in (0,T]}\mathcal{I}_{\lambda}\left((u_1^{\infty})_t\right)$. Setting
 $$
   \gamma_0(t)=\left\{\begin{array}{ll}
    (u_1^{\infty})_{(tT)}, \ \ &\mbox{for} \ t>0,\\
    0, \ \ & \mbox{for} \ t=0.
    \end{array}\right.
 $$
 Then $\gamma_0\in \Gamma$ defined by Lemma \ref{lem 3.4} (ii). Moreover
 \begin{equation}\label{N50}
   \mathcal{I}_{\lambda} \left((u_1^{\infty})_{t_{\lambda}}\right)=\max_{t\in [0,1]}\mathcal{I}_{\lambda}\left(\gamma_0(t)\right)
       \ge c_{\lambda}.
 \end{equation}
 Let
 \begin{equation}\label{G56}
   \zeta_0:=\min\{3\bar{r}/8(1+|\bar{x}|), 1/4\}.
 \end{equation}
 Then it follows from \eqref{G56} that
 \begin{equation}\label{G58}
   |x-\bar{x}|\le \frac{\bar{r}}{2} \ \ \mbox{and} \ \ s\in [1-\zeta_0, 1+\zeta_0] \Rightarrow |sx-\bar{x}|\le \bar{r}.
 \end{equation}
 Since $\mathcal{P}^{\infty}(u_1^{\infty})=0$, then $\int_{\R^N}F(u_1^{\infty})\mathrm{d}x>0$. Let
 \begin{eqnarray*}\label{N59}
   \bar{\lambda}
    & :=  & \max\left\{\frac{1}{2}, 1-\frac{\min_{s\in [1-\zeta_0, 1+\zeta_0]}\int_{\R^3}\left[V_{\infty}-
              V(sx)\right]|u_1^{\infty}|^2\mathrm{d}x}{2(1-\zeta_0)^{-3}T^3\int_{\R^3}F(u_1^{\infty})\mathrm{d}x},1-\frac{b\zeta_0^2\|\nabla u_1^{\infty}\|_2^4}{12T^3\int_{\R^3}F(u_1^{\infty})\mathrm{d}x}\right\} \ \ \ \ \ \
 \end{eqnarray*}
 Then it follows from \eqref{N48} and \eqref{G58} that $1/2\le \bar{\lambda}<1$. We have two cases to distinguish:

 \vskip2mm
 \par
   Case i). $t_{\lambda}\in [1-\zeta_0, 1+\zeta_0]$. From \eqref{Ilu}, \eqref{Il}, \eqref{N45}-\eqref{N50}, \eqref{G56} and
 Lemma \ref{lem 3.4} (iv), we have
 \begin{eqnarray*}
   m_{\lambda}^{\infty}
    & \ge & m_1^{\infty}=\mathcal{I}_1^{\infty}(u_1^{\infty})\ge \mathcal{I}_1^{\infty}\left((u_1^{\infty})_{t_{\lambda}}\right)\nonumber\\
    &  =  & \mathcal{I}_{\lambda}\left((u_1^{\infty})_{t_{\lambda}}\right)
              -(1-\lambda)t_{\lambda}^3\int_{\R^3}F(u_1^{\infty})\mathrm{d}x +\frac{t_{\lambda}^3}{2}\int_{\R^3}[V_{\infty}-V(t_{\lambda}x)]|u_1^{\infty}|^2\mathrm{d}x\nonumber\\
    & \ge & c_{\lambda} -(1-\lambda)T^3\int_{\R^3}F(u_1^{\infty})\mathrm{d}x\nonumber\\
    &     & \ \     +\frac{(1-\zeta_0)^3}{2}\min_{s\in [1-\zeta_0, 1+\zeta_0]}
                \int_{\R^3}\left[V_{\infty}-V(sx)\right]|u_1^{\infty}|^2\mathrm{d}x\nonumber\\
    &  >  & c_{\lambda}, \ \ \ \ \forall \ \lambda\in (\bar{\lambda}, 1].
 \end{eqnarray*}
 \par
   Case ii). $t_{\lambda}\in (0, 1-\zeta_0)\cup (1+\zeta_0, T]$. From \eqref{Ilu}, \eqref{Il}, \eqref{N45}, \eqref{N47},
 \eqref{N50}, \eqref{G56} and Lemma \ref{lem 3.4} (iv), we have
 \begin{eqnarray*}
   m_{\lambda}^{\infty}
    & \ge & m_1^{\infty}=\mathcal{I}_1^{\infty}(u_1^{\infty})\ge \mathcal{I}_1^{\infty}\left((u_1^{\infty})_{t_{\lambda}}\right)
             +\frac{b(1-t_{\lambda})^2(1+2t_{\lambda})}{12}\|\nabla u_1^{\infty}\|_2^4\nonumber\\
    &  =  & \mathcal{I}_{\lambda}\left((u_1^{\infty})_{t_{\lambda}}\right)
              -(1-\lambda)t_{\lambda}^3\int_{\R^3}F(u_1^{\infty})\mathrm{d}x\nonumber\\
    &     & \ \  +\frac{t_{\lambda}^3}{2}\int_{\R^3}[V_{\infty}-V(t_{\lambda}x)]|u_1^{\infty}|^2\mathrm{d}x
              +\frac{b(1-t_{\lambda})^2(1+2t_{\lambda})}{12}\|\nabla u_1^{\infty}\|_2^4\nonumber\\
    & \ge & c_{\lambda} -(1-\lambda)T^3\int_{\R^3}F(u_1^{\infty})\mathrm{d}x
              +\frac{b\zeta_0^2}{12}\|\nabla u_1^{\infty}\|_2^4\nonumber\\
    &  >  & c_{\lambda}, \ \ \ \ \forall \ \lambda\in (\bar{\lambda}, 1].
 \end{eqnarray*}
 In both cases, we obtain that $c_{\lambda}<m_{\lambda}^{\infty}$ for $\lambda\in (\bar{\lambda}, 1]$.
 \end{proof}

 \begin{lemma}\label{lem 3.6}
 Assume that {\rm (V1)-(V3)} and {\rm (F1)-(F3)} hold. Let $\{u_n\}$ be a bounded (PS)$_{c_{\lambda}}$
 sequence for $\mathcal{I}_{\lambda}$ with $\lambda\in  [1/2, 1]$. Then there exist a subsequence of $\{u_n\}$, still denoted by
 $\{u_n\}$, an integer $l\in \N \cup \{0\}$, and $u_{\lambda}\in H^1(\R^3)$ such that
 \begin{enumerate}[{\rm (i)}]
  \item $A_{\lambda}^2:=\lim_{n\to\infty}\|\nabla u_n\|_2^2$ exists, $u_n\rightharpoonup u_{\lambda}$ in $H^1(\R^3)$ and $\mathcal{E}_{\lambda}'(u_{\lambda})=0$;
  \item $w^k\ne 0$ and $(\mathcal{E}_{\lambda}^{\infty})'(w^k)=0$ for $1\le k\le l$;
 \item
 \begin{equation}\label{N60}
    c+\frac{bA_{\lambda}^4}{4}= \mathcal{E}_{\lambda}(u_{\lambda})+\sum_{k=1}^{l}\mathcal{ E}_{\lambda}^{\infty}(w^k);
 \end{equation}
 \begin{equation}\label{N61}
   A_{\lambda}^2= \|\nabla u_{\lambda}\|_2^2+\sum_{k=1}^{l}\|\nabla w^k\|_2^2,
 \end{equation}
 where
 \begin{equation}\label{Up}
  \mathcal{E}_{\lambda}(u) = \frac{a+bA_{\lambda}^2}{2}\int_{\R^3}|\nabla u|^2\mathrm{d}x
                   +\frac{1}{2}\int_{\R^3}V(x)u^2\mathrm{d}x-\lambda\int_{\R^3}F(u)\mathrm{d}x
 \end{equation}
 and
 \begin{equation}\label{Upi}
   \mathcal{E}_{\lambda}^{\infty}(u) = \frac{a+bA_{\lambda}^2}{2}\int_{\R^3}|\nabla u|^2\mathrm{d}x
                   +\frac{V_{\infty}}{2}\int_{\R^3}u^2\mathrm{d}x-\lambda\int_{\R^3}F(u)\mathrm{d}x.
 \end{equation}
 \end{enumerate}
 where we agree that in the case $l = 0$ the above holds without $w^k$.
 \end{lemma}

    Analogous to the proof of \cite[Lemma 3.4]{LG}, we can prove Lemma \ref{lem 3.6}, so we omit it here.

 \begin{lemma}\label{lem 3.7}
 Assume that {\rm (V1)-(V3)} and {\rm (F1)-(F3)} hold. Then for almost every
 $\lambda\in (\bar{\lambda},1]$, there exists $u_{\lambda}\in H^1(\R^3)\setminus \{0\}$ such that
 \begin{equation}\label{D31}
   \mathcal{I}_{\lambda}'(u_{\lambda})=0, \ \ \ \ \mathcal{I}_{\lambda}(u_{\lambda}) = c_{\lambda}.
 \end{equation}
\end{lemma}

\begin{proof} Under (V1)-(V3) and (F1)-(F3), Lemma \ref{lem 3.4} implies that $\mathcal{I}_{\lambda}(u)$ satisfies the assumptions of
 Proposition \ref{pro 3.1} with $X=H^1(\R^3)$ and $\Phi_{\lambda}=\mathcal{I}_{\lambda}$. So for almost every $\lambda\in [0.5,1]$,
 there exists a bounded sequence $\{u_n(\lambda)\} \subset H^1(\R^3)$ (for simplicity, we denote the sequence by $\{u_n\}$
 instead of $\{u_n(\lambda)\}$) such that
 \begin{equation}\label{PS}
   \mathcal{I}_{\lambda}(u_n)\rightarrow c_{\lambda}>0, \ \ \ \ \|\mathcal{I}_{\lambda}'(u_n)\| \rightarrow 0.
 \end{equation}
 By Lemma \ref{lem 3.6}, there exist a subsequence of $\{u_n\}$, still denoted by $\{u_n\}$, and $u_{\lambda}\in H^1(\R^3)$ such
 that $A_{\lambda}^2:=\lim_{n\to\infty}\|\nabla u_n\|_2^2$ exists, $u_n\rightharpoonup u_{\lambda}$ in $H^1(\R^3)$ and
 $\mathcal{E}_{\lambda} '(u_{\lambda})=0$, and there exist $l\in \N$ and $w^1, \ldots, w^l\in H^1(\R^3)\setminus \{0\}$ such that
 $(\mathcal{E}_{\lambda}^{\infty})'(w^k)=0$ for $1\le k\le l$,
 \begin{equation}\label{Ab1}
    c_{\lambda}+\frac{bA_{\lambda}^4}{4}= \mathcal{E}_{\lambda}(u_{\lambda})+\sum_{k=1}^{l}\mathcal{E}_{\lambda}^{\infty}(w^k)
 \end{equation}
 and
 \begin{equation}\label{Ab2}
   A_{\lambda}^2= \|\nabla u_{\lambda}\|_2^2+\sum_{k=1}^{l}\|\nabla w^k\|_2^2.
 \end{equation}
 Since $\mathcal{E}_{\lambda}'(u_{\lambda})=0$, then we have the Poho\u zaev identity referred to the functional $\mathcal{E}_{\lambda}$
 \begin{equation}\label{Plt}
  \tilde{\mathcal{P}}_{\lambda}(u_{\lambda})
     :=  \frac{a+bA_{\lambda}^2}{2}\|\nabla u_{\lambda}\|_2^2+\frac{1}{2}\int_{{\R}^3}\left[3V(x)+\nabla V(x)\cdot x\right]u_{\lambda}^2\mathrm{d}x
         -3\lambda\int_{{\R}^3}F(u_{\lambda})\mathrm{d}x=0.
 \end{equation}
 From (V3) and Hardy inequality
 \begin{equation}\label{D32}
   a\|\nabla u_{\lambda}\|_2^2 \ge \frac{a}{4}\int_{\R^3}\frac{u_{\lambda}^2}{|x|^2}\mathrm{d}x
    \ge \int_{{\R}^3}\nabla V(x)\cdot xu_{\lambda}^2\mathrm{d}x.
 \end{equation}
 It follows from \eqref{Up}, \eqref{Plt} and \eqref{D32} that
 \begin{eqnarray}\label{D33}
   \mathcal{E}_{\lambda}(u_{\lambda})
     &  =  & \mathcal{E}_{\lambda}(u_{\lambda})-\frac{1}{3}\tilde{\mathcal{P}}_{\lambda}(u_{\lambda})\nonumber\\
     &  =  & \frac{a+bA_{\lambda}^2}{3}\|\nabla u_{\lambda}\|_2^2-\frac{1}{6}\int_{{\R}^3}\nabla V(x)\cdot xu_{\lambda}^2\mathrm{d}x\nonumber\\
     & \ge & \frac{bA_{\lambda}^2}{3}\|\nabla u_{\lambda}\|_2^2.
 \end{eqnarray}
 Since $(\mathcal{E}_{\lambda}^{\infty})'(w^k)=0$, then we have the Poho\u zaev identity referred to the functional $\mathcal{E}_{\lambda}^{\infty}$
 \begin{equation}\label{D34}
  \tilde{\mathcal{P}}_{\lambda}^{\infty}(w^k)
     :=  \frac{a+bA_{\lambda}^2}{2}\|\nabla w^k\|_2^2+\frac{3V_{\infty}}{2}\int_{{\R}^3}(w^k)^2\mathrm{d}x
         -3\lambda\int_{{\R}^3}F(w^k)\mathrm{d}x=0.
 \end{equation}
 Thus, from \eqref{Ab2} and \eqref{D34}, we have
 \begin{eqnarray}\label{D35}
   0=\tilde{\mathcal{P}}_{\lambda}^{\infty}(w^k) \ge \mathcal{P}_{\lambda}^{\infty}(w^k).
 \end{eqnarray}
 Since $w^k\ne 0$ and $w^k\in\Lambda$, in view of Lemmas \ref{lem 2.6} and \ref{lem 2.7}, there exists $t_k>0$ such that
 $(w^k)_{t_k}\in \mathcal{M}_{\lambda}^{\infty}$.
 From \eqref{Il}, \eqref{JlL}, \eqref{N45}, \eqref{Upi}, \eqref{Ab2}, \eqref{D34} and \eqref{D35}, one has
 \begin{eqnarray}\label{D36}
   \mathcal{E}_{\lambda}^{\infty}(w^k)
     &  =  & \mathcal{E}_{\lambda}^{\infty}(w^k)-\frac{1}{3}\tilde{\mathcal{P}}_{\lambda}^{\infty}(w^k)
              =\frac{a+bA_{\lambda}^2}{3}\|\nabla w^k\|_2^2\nonumber\\
     &  =  & \frac{bA_{\lambda}^2}{3}\|\nabla w^k\|_2^2+\mathcal{I}_{\lambda}^{\infty}(w^k)
               -\frac{1}{3}\mathcal{P}_{\lambda}^{\infty}(w^k)-\frac{b}{12}\|\nabla w^k\|_2^4\nonumber\\
     & \ge & \frac{bA_{\lambda}^2}{4}\|\nabla w^k\|_2^2+\mathcal{I}_{\lambda}^{\infty}\left((w^k)_{t_k}\right)
             -\frac{t_k^3}{3}\mathcal{P}_{\lambda}^{\infty}(w^k)\nonumber\\
     & \ge & \frac{bA_{\lambda}^2}{4}\|\nabla w^k\|_2^2+m_{\lambda}^{\infty}.
 \end{eqnarray}
 It follows from \eqref{Ab1}, \eqref{Ab2}, \eqref{D33} and \eqref{D36} that
 \begin{eqnarray*}
   c_{\lambda}+\frac{bA_{\lambda}^4}{4}
    &  =  & \mathcal{E}_{\lambda}(u_{\lambda})+\sum_{k=1}^{l}\mathcal{E}_{\lambda}^{\infty}(w^k)\\
    & \ge & lm_{\lambda}^{\infty}+\frac{bA_{\lambda}^2}{4}\left[\|\nabla u_{\lambda}\|_2^2+\sum_{k=1}^{l}\|\nabla w^k\|_2^2\right]\\
    &  =  & lm_{\lambda}^{\infty}+\frac{bA_{\lambda}^4}{4}, \ \ \ \ \forall \ \lambda\in (\bar{\lambda}, 1],
 \end{eqnarray*}
 which, together with Lemma \ref{lem 3.5}, implies that $l=0$ and $\mathcal{E}_{\lambda}(u_{\lambda})=c_{\lambda}+\frac{bA_{\lambda}^4}{4}$.
 Hence, it follows from \eqref{Ab2} that $A_{\lambda}=\|\nabla u_{\lambda}\|_2$, and so $\mathcal{I}_{\lambda}'(u_{\lambda}) = 0$
 and $\mathcal{I}_{\lambda}(u_{\lambda}) = c_{\lambda}$.
 \end{proof}

 \begin{proof}[Proof of Theorem  \ref{thm1.1}] In view of Lemma \ref{lem 3.7}, there exist two sequences
 $\{\lambda_n\}\subset (\bar{\lambda}, 1]$ and $\{u_{\lambda_n}\}\subset H^1(\R^3)$, denoted by $\{u_n\}$, such that
 \begin{equation}\label{Q00}
   \lambda_n\rightarrow 1, \ \ \ \ \mathcal{I}_{\lambda_n}'(u_n)=0, \ \ \ \ \mathcal{I}_{\lambda_n}(u_n) = c_{\lambda_n}.
 \end{equation}
 From (V3), \eqref{Ilu}, \eqref{B22}, \eqref{Q00} and Lemma \ref{lem 3.4} (v), one has
 \begin{eqnarray}\label{Q01}
   c_{1}
    & \ge & c_{\lambda_n}=\mathcal{I}_{\lambda_n}(u_n)-\frac{1}{3}\mathcal{P}_{\lambda_n}(u_n)\nonumber\\
    &  =  & \frac{a}{3}\|\nabla u_n\|_2^2-\frac{1}{6}\int_{{\R}^3}\nabla V(x)\cdot xu_n^2\mathrm{d}x
              +\frac{b}{6}\|\nabla u_n\|_2^4\nonumber\\
    & \ge & \frac{b}{6}\|\nabla u_n\|_2^4.
 \end{eqnarray}
 This shows that $\{\|\nabla u_n\|_2\}$ is bounded.  Next, we demonstrate that $\{u_n\}$ is bounded in $H^1(\R^3)$.
 By (V1), (V2), (F1), (F2), \eqref{Ilu}, \eqref{Q00}, \eqref{Q01} and the Sobolev embedding inequality, we have
 \begin{eqnarray*}\label{Q03}
   \gamma_1'\|u_n\|^2
    & \le & \int_{{\R}^3}\left[a|\nabla u_n|^2+V(x)u_n^2\right]\mathrm{d}x\\
    & \le & 2c_{\lambda_n}+2\lambda_n\int_{{\R}^3}F(u_n)\mathrm{d}x\\
    & \le & 2c_{1}+\frac{\gamma_1'}{4}\|u_n\|^2+C_5\|u_n\|_6^6\nonumber\\
   & \le & 2c_{1}+\frac{\gamma_1'}{4}\|u_n\|_2^2+C_5S^{-3}\|\nabla u_n\|_2^6,
 \end{eqnarray*}
 where $\gamma_1'$ is a positive constant. Hence, $\{u_n\}$ is bounded in $H^1(\R^3)$. In view of Lemma \ref{lem 3.4} (v), we have $\lim_{n\to\infty}c_{\lambda_n}=c_*\le c_1$. Hence,
 it follows from \eqref{IU}, \eqref{Ilu} and \eqref{Q00} that
 \begin{equation}\label{Q04}
   \mathcal{I}(u_n)\rightarrow  c_*, \ \ \ \  \mathcal{I}'(u_n)\rightarrow 0.
 \end{equation}
 This shows that $\{u_n\}$ satisfy \eqref{PS} with $\mathcal{I}_{\lambda}=\mathcal{I}$ and $c_{\lambda}=c_*$. In view of
 the proof of Lemma \ref{lem 3.7}, we can show that there exists $\tilde{u}\in H^1(\R^N)\setminus \{0\}$ such that
 \begin{equation}\label{Q05}
   \mathcal{I}'(\tilde{u})=0, \ \ \ \ 0< \mathcal{I}(\tilde{u})\le c_1.
 \end{equation}
 Let
 $$
   \mathcal{K}:=\left\{u\in H^1(\R^N)\setminus \{0\} : \mathcal{I}'(u)=0\right\}, \ \ \ \ \hat{m}:=\inf_{u\in\mathcal{K}}\mathcal{I}(u).
 $$
 Then \eqref{Q05} shows that $\mathcal{K}\ne \emptyset$ and $\hat{m}\le c_1$. For any $u\in \mathcal{K}$, Lemma \ref{lem 3.2}
 implies $\mathcal{P}(u)=\mathcal{P}_1(u)=0$. Hence it follows from \eqref{D33} that $\mathcal{I}(u)=\mathcal{I}_1(u)>0$, and so $\hat{m}\ge 0$.
 Let $\{u_n\}\subset \mathcal{K}$ such that
 \begin{equation}\label{Z01}
   \mathcal{I}'(u_n)=0, \ \ \ \ \mathcal{I}(u_n) \rightarrow \hat{m}.
 \end{equation}
 In view of Lemma \ref{lem 3.5}, $\hat{m}\le c_1<m_1^{\infty}$. By a similar argument as in the proof of Lemma \ref{lem 3.7}, we can prove
 that there exists $\bar{u}\in H^1(\R^N)\setminus \{0\}$ such that
 \begin{equation}\label{Z02}
   \mathcal{I}'(\bar{u})=0, \ \ \ \ \mathcal{I}(\bar{u}) = \hat{m}.
 \end{equation}
 This shows that $\bar{u}$ is a nontrivial least energy solution of \eqref{KE}.
 \end{proof}

\section*{Acknowledgements}
This work was partially supported by the National Natural Science Foundation of China (11571370).

\end{document}